\documentclass[prd,onecolumn,nopacs,nofootinbib]{revtex4}

\usepackage{amsmath}
\usepackage{graphicx}
\usepackage{amsfonts}
\usepackage{latexsym}
\usepackage{bbold}
\usepackage{wasysym}
\usepackage{calligra}
\usepackage{float}
\usepackage{ulem}
\usepackage{inputenc}
\usepackage{xspace}
\usepackage{url}
\usepackage{epstopdf}
\usepackage{tikz}
\usepackage{amssymb}
\usepackage{enumitem}

\usepackage{amsmath}
\usepackage{amssymb}
\usepackage{graphicx}
\usepackage{amsfonts}
\usepackage{latexsym}
\usepackage{bbold}
\usepackage{wasysym}
\usepackage{calligra}
\usepackage{float}
\usepackage{ulem}
\usepackage{inputenc}
\usepackage{xspace}
\usepackage{url}
\usepackage{epstopdf}
\usepackage{tikz}
\usepackage{amsthm}
\usepackage{soul}

\newtheorem{mydef}{Definition}
\newtheorem{cor}{Corollary}

\newtheorem{myprop}{Proposition}
\newtheorem{mylemma}{Lemma}
\newtheorem{theorem}{Theorem}

\newcommand{\be}{\begin{equation}}
\newcommand{\ee}{\end{equation}}
\newcommand{\beq} {\begin{equation}}
\newcommand{\eeq} {\end{equation}}
\newcommand{\ba}{\begin{eqnarray}}
\newcommand{\ea}{\end{eqnarray}}

\begin{document}

	\title{	On a Torsion/Curvature Analogue of Dual Connections and Statistical Manifolds}
	%\title{On Dual Connections, Statistical Manifolds and Torsion}
	\author{Damianos Iosifidis}
	\affiliation{Laboratory of Theoretical Physics, Institute of Physics, University of Tartu, W. Ostwaldi 1, 50411 Tartu, Estonia.}
	\email{damianos.iosifidis@ut.ee}

	\date{\today}
	\begin{abstract}
		
	In analogy with the concept of a non-metric dual connection, which is essential in defining statistical manifolds, we develop that of a torsion dual connection. Consequently, we illustrate the geometrical meaning of such a torsion dual connection and show how the use of both connections  preserves the cracking of parallelograms in spaces equipped with a connection and its torsion dual. The coefficients of such a torsion dual connection are essentially computed by demanding a vanishing mutual torsion among the two connections. For this manifold we then prove two basic Theorems. In particular, if both connections are metric-compatible we show that there exists a specific $3$-form measuring how the connection and its torsion dual deviate away from the Levi-Civita one. Furthermore, we prove that for these torsion dual manifolds flatness of one connection does not necessary impose flatness on the other but rather that the curvature tensor of the latter is given by a specific divergence. Finally, we give a self-consistent definition of the mutual curvature tensor of two connections and subsequently define the notion of a curvature dual connection.
		
	\end{abstract}
	
	\maketitle
	
	\allowdisplaybreaks
	
	%\newpage
	
	\tableofcontents
	%\newpage

\section{Introduction}

Dual Connections \cite{amari1985differential,lauritzen1987statistical} are at the core of Information-Geometry \cite{amari1997information}. They are essential for the definition of 'Statistical Manifolds' \cite{amari1987differential} which despite of their name are purely geometric constructions  that find many applications even beyond statistics \cite{amari2000methods}. Geometrically speaking, given an affine connection $\nabla$, its dual connection $\nabla^{\star}$ is defined in such as way so as to preserve inner products of vectors given that one is transported with the former and the other with the latter \cite{amari1985differential}. These connections are non-metric but torsion-free\footnote{However, there are extensions of statistical manifolds admitting  torsion \cite{kurose2007statistical,matsuzoe2010statistical}. These are encountered in quantum statistical models \cite{amari2000methods}.} and there exists a totally symmetric tensor, called the cubic tensor, which is essentially a totally symmetric non-metricity and quantifies how much the geometries deviate away from the Riemannian one.

One of the consequences of non-metricity is that the inner products of vectors are no longer preserved. Then, the definition of the dual connection rests precisely on the demand that when one vector is transported with $\nabla$ and the other with respect to its dual $\nabla^{\star}$ the inner product will be preserved. In other words, given $\nabla$ the dual connection is obtained by demanding such an inner product to be preserved. It is then natural to ask if we can have an analogous definition for two connections with torsion. To be more precise, let us recall that one of the effects of torsion is the inability to form infinitesimal parallelograms. Indeed, if we take to curves with their associated tangent vectors and transport each one in the direction of the other (see example in section $IV$ for more details) the end result would be a pentagon if torsion is there. Namely, the vectors have twisted due to torsion. The question we now raise is if we can define a 'torsion dual connection' that would 'cooperate'
 with $\nabla$ in order to keep the parallelogram from breaking to pentagon when one vector is transported with respect to $\nabla$ and the other with respect to its torsion dual $\nabla^{\star}$. As we shall show this is indeed possible and furthermore such a demand uniquely specifies all the coefficients of the dual connection in terms of the coefficients of $\nabla$. In addition the mean connection, formed by $\nabla$ and its torsion dual, is the Levi-Civita  and duality is involutive (i.e. $(\nabla^{\star})^{\star}=\nabla$), as was also the case for usual non-metric dual connections.

Having defined this 'torsion dual connection' we then derive two basic Theorems analogous to some fundamental Theorems of Information Geometry. Firstly, we show that if the connection and its torsion dual are both metric then there exists a $3$-form $A$ such that the connection and its dual take the form $\nabla=\nabla^{(0)}+A$ and $\nabla=\nabla^{(0)}-A$ respectively. This totally antisymmetric tensor $A$ is the analogue of the totally symmetric cubic tensor that appears in statistical manifolds. Our second Theorem involves the curvatures of the two connections. In particular if the connection and its torsion dual are metric  we prove that $\nabla$-flatness does not imply $\nabla^{\star}$-flatness and vice versa as was the case for statistical manifolds. Our result rather shows that flatness of one of the connections, either $\nabla$ or its torsion dual, implies that the curvature of the other is given by a specific divergence-like term. However, a weaker constraint on the scalar curvature of the two connections is obtained, namely if $\nabla$ is scalar-flat then so is $\nabla^{\star}$ and also the other way around.

We then show that the torsion dual connection appears as a consequence of a vanishing mutual torsion of two connections, and in the same way the 'classical' dual connection used in information geometry is derived by demanding a vanishing mutual non-metricity (to be defined in section $III$) of the two connections. Finally, guided by this last observation after giving a self-consistent definition\footnote{As we shall show, other definitions of the mutual curvature that are found in the literature are rather unsatisfactory. } of the mutual curvature of two connections we define what we may call a "Curvature Dual Connection" by demanding a vanishing mutual curvature of the two connections.

The paper is organized as follows. Firstly we set up the conventions and notation and define some essential geometrical objects we are going to be using throughout. Then, we briefly review the concept of dual connections and statistical manifolds and also obtain some well known results using different methods which may provide new insights and a different fresh  look at some old problems. Consequently, in section $IV$ we present the main results of this study by first introducing the concept of the torsion dual connection, discussing some of its properties, and then proving two basic Theorems associated to this construction. Finally, in section $V$  we give a novel definition for the mutual curvature of two connections and with this we also define the notion of a curvature dual connection. We then conclude our results with a discussion.

\section{Geometry}
% A finite linear combination whose coefficients add up to 1 is called a convex
%linear combination
\subsection{Intro}

We consider an $n$-dimensional manifold $\mathcal{M}$ equipped with an affine connection $\nabla$. Recall that an affine connection on $\mathcal{M}$ is an $\mathbb{R}$-bilinear map $\nabla:$ $\mathfrak{X}(M)$ $\times$ $\mathfrak{X}(M)$ $ \rightarrow \mathfrak{X}(M) $
written $\nabla_{X}Y$ for $\nabla(X,Y)$ with the following two properties\footnote{Here $\mathcal{F}$ denotes the ring  $\mathcal{C}^{\infty}(M)$  of $\mathcal{C}^{\infty}$ functions on $\mathcal{M}$ and $X,Y$ $\in$  $\mathfrak{X}(M)$ with the later symbol denoting the set of $C^{\infty}$ vector fields on $\mathcal{M}$ . }
\begin{enumerate}
    \item $\nabla_{fX}Y = f \nabla_{X}{Y} $
    \item  $\nabla_{X}(fY)=(Xf)Y+ f\nabla_{X}Y$
\end{enumerate}
with the former expressing $\mathcal{F}$-linearity in $X$ and the later being the Leibniz rule. Note that the connection is not $\mathcal{F}$-linear in $Y$.

Having introduced an affine connection, and without the need of any metric, the curvature and and torsion of $\nabla$, over the manifold,  are given by
\beq
R(X,Y)Z:=\nabla_{X}\nabla_{Y}Z-\nabla_{Y}\nabla_{X}Z-\nabla_{[X,Y]}Z
\eeq
and
\beq
T(X,Y):=\nabla_{X}Y-\nabla_{Y}X-[X,Y]
\eeq
where $X,Y,Z$ are $C^{\infty}$ vector fields. These are constructed in such a way so as to be $\mathcal{F}$-linear in all their arguments, i.e. they are tensors.
In local coordinates we have the defining relation\footnote{Beware of our convention for the index placing in the connection coefficients.  }
\beq
\nabla_{\partial_{i}}\partial_{j}=\Gamma^{k}_{\;\; ji}\partial_{k} \label{concof}
\eeq
where the $\Gamma^{k}_{\;\;ji}$ are called the Christoffel symbols or just the connection coefficients. By abusing terminology, we will frequently refer to the connection coefficients simply as the connection, since this is quite common in the gravity community. Using this, we can derive the form of curvature and torsion in local coordinates, which in our conventions read\footnote{Note that in these conventions we have that $R(\partial_{i},\partial_{j})\partial_{k}=R^{m}_{\;\; kij}\partial_{m}$ and also $T(\partial_{i},\partial_{j})=T^{m}_{\;\; ij}\partial_{m}$.} 
 \beq
    R^{m}_{\;\; ijk}:=2 \partial_{[j} \Gamma_{\;\; |i|k]}^{m}+2 \Gamma^{m}_{\;\; l[j}\Gamma^{l}_{\;\; |i|k]} \label{Riem}
    \eeq
    and
   \beq
   T^{m}_{\;\; ij}:=\Gamma^{m}_{\;\; ji}-\Gamma^{m}_{\;\; ij}
    \eeq
respectively. In what follows shall also use the redefined torsion coefficients
\beq
S_{ij}^{\;\;\; k}:=-\frac{1}{2}T^{k}_{\;\; ij}=\Gamma^{k}_{\;\; [ij]}
\eeq
which are  somewhat more convenient for calculations. 

If one wants to define inner products, norms etc. on $\mathcal{M}$ then, in addition, one needs to endow the space with a metric $g$ such that
\beq
g(X,Y)=<X,Y>:=g_{ij}X^{i}Y^{j}
\eeq
where $X,Y \in \mathfrak{X}(M)$ . In local coordinates the metric components are given by
\beq
g(\partial_{i},\partial_{j})=g_{ij}
\eeq
Then with the additional concept of the metric one can define the non-metricity of a connection $\nabla$ according to
\beq
Q(Z,X,Y):=(\nabla_{Z}g)(X,Y)=\nabla_{Z} g(X,Y)
\eeq
which measures the failure of the connection to be covariantly conserved and has symmetry in the exchange of $X$ and $Y$. In local coordinates it is given as
\beq
Q_{ijk}:=\nabla_{i}g_{jk}
\eeq
With the above ingredients it is not difficult to show that a general affine connection $\nabla$ on $M$ with coefficients $\Gamma^{m}_{\;\;ij} $can be decomposed as
\beq
\Gamma^{m}_{\;\; ij}=\Gamma^{m\;(0)}_{\;\; ij}+N^{m}_{\;\; ij} \label{expans}
\eeq
where $\Gamma^{(0)m}_{\;\; ij}$ is the usual torsionless and metric Levi-Civita connection and 
\beq
N^{m}_{\;\; ij}=\frac{1}{2}g^{ml}(Q_{lij}-Q_{jli}-Q_{ijl})-g^{ml}(S_{lij}+S_{lji}-S_{ijl})
\eeq
which is oftentimes referred to as the distortion tensor describing the non-Riemannian features of the manifold. In addition the decomposition (\ref{expans}) is called the post-Riemannian expansion of the connection. Finally, plugging (\ref{expans}) into the definition of the curvature tensor we get the post-Riemannian expansion of the latter as
\beq
R^{m}_{\;\;jkl}=R^{m\;(0)}_{\;\;jkl}+2 \nabla^{(0)}_{[k}N^{m}_{\;\;|j|l]}+2 N^{m}_{\;\; n[k}N^{n}_{\;|j|l]} \label{Riemdec}
\eeq
where $R^{m}_{\;\;jkl}=R^{m}_{\;\;jkl}(\nabla)$ are the coefficients of the curvature tensor of the full connection $\nabla$ and  $R^{m\;(0)}_{\;\;jkl}=R^{m}_{\;\;jkl}(\nabla^{(0)})$ is the Riemann tensor associated with the Levi-Civita connection $\nabla^{(0)}$. The above expansion will be essential for the proof of our main Theorem which we present in section (\ref{sec}).

\subsection{Endowing $\mathcal{M}$ with two connections}

Given two affine connections $\nabla^{(1)}$ and $\nabla^{(2)}$, it is trivial to show that their sum is not a connection. However the convex linear combination
\beq 
\nabla=t\nabla^{(1)}+(1-t)\nabla^{(2)}\;\;, \; t \in \mathbb{R}
\eeq
with connection coefficients 
\beq
\Gamma^{i}_{\;\;\;jk}=t\Gamma^{i \; (1)}_{\;\;\;jk}+(1-t)\Gamma^{i \; (2)}_{\;\;\;jk} \label{Gamma}
\eeq
does define a new connection $\nabla$ on $\mathcal{M}$.
Note also that the difference of two affine connections always defines legal tensors, as it can be trivially shown. As a result for the above two connections, we define the difference tensor through
\beq
K(Y,X):=\nabla_{X}^{(1)}Y-\nabla_{X}^{(2)}Y \label{diften0}
\eeq
which is expressed in local coordinates as\footnote{Note the somewhat unconventional position of $X$ and $Y$ in the definition (\ref{diften0}). This is a remnant of the index placing of the connection coefficients in (\ref{concof}). }
\beq
K^{l}_{\;\;ij}:=\Gamma^{l \; (1)}_{\;\;\;ij}-\Gamma^{l \; (2)}_{\;\;\;ij} \label{diften}
\eeq
Furthermore, this tensor appears quite naturally by forming the commutator of the two connections with respect to their connection labeling index and acting it on a scalar.  Indeed, just as the torsion tensor (in components) for a single connection can be defined through its action on a $\mathcal{C}^{\infty}$ function $f$, as
\beq
(\nabla_{i}\nabla_{j}-\nabla_{j}\nabla_{i})f=- T^{l}_{\;\; ij}
 \nabla_{l}f
\eeq
Similarly, given the two connections above, taking the aforementioned commutator between them and acting it on a $\mathcal{C}^{\infty}$ scalar function we readily find
\beq
\Big( \nabla_{i}^{(1)}\nabla_{j}^{(2)}-\nabla_{i}^{(2)}\nabla_{j}^{(1)}\Big)f=-\Big(\Gamma^{l \;(1)}_{\;\;\;ji}-\Gamma^{l \; (2)}_{\;\;\;ji})\Big)\partial_{\lambda}f \equiv- K^{l}_{\;\;ji}\partial_{\lambda} f
\eeq

Given, the above expression (\ref{Riem}) of the Riemann tensor of $\nabla$, and using the relation (\ref{Gamma}) an elementary calculation reveals that
\beq
R^{m}_{\;\;ijk}(\nabla)=tR^{m}_{\;\;ijk}(\nabla^{(1)})+(1-t)R^{m}_{\;\;ijk}(\nabla^{(2)})-2t(1-t)K^{m}_{\;\;l[j}K^{l}_{\;\;|i|k]} \label{Riem2}
\eeq
 Let us stress out that, of course, the latter is not just a linear combination of the Riemann tensors of the two connections. It is  interesting however to note that the additional piece, which is tensorial as it should, is built entirely out of the difference tensor of the two connections as defined by ($\ref{diften}$). 

 Contracting twice in the usual manner the above Riemann tensor we obtain for the Ricci scalar\footnote{We define the Ricci scalar as usual, namely $Ric:=R^{m}_{\;\;imj}g^{ij}$}
\beq
Ric(\nabla)=t Ric(\nabla^{(1)})+(1-t)Ric(\nabla^{(2)})-t(1-t)\Big( K^{i}_{\;\;ji}K^{j}_{\;\; kl}g^{kl}-K^{ijk}K_{jki}\Big) \label{R}
\eeq
Again, the Ricci scalar of the full connection is not just a weighted sum of the Ricci scalars of the two connections but we have interesting couplings between the associated connections given in the form of quadratic invariants of the difference tensor (last two terms appearing in the parenthesis on the right-hand side of ($\ref{R}$)).

\section{Dual Connections and Statistical Manifolds}

 The concept of dual connection was first introduced by Amari \cite{amari1985differential} and later by Lauritzen  \cite{lauritzen1987statistical} and was further developed and applied by Amari \cite{amari1987differential,amari1997information,amari2000methods}. We recall its definition below.
 \begin{mydef}{Dual Connections.}
 Given a connection $\nabla$ and a metric $g$ on a manifold $\mathcal{M}$, namely given the  triplet (g,$\nabla$,$\mathcal{M}$), the  dual-connection $\nabla^{\star}$ is defined through the relation \cite{amari1987differential}
\beq
Z(g(X,Y))=g(\nabla_{Z}X,Y)+g(X,\nabla^{\star}_{Z}Y) \label{metdual}
\eeq
$\forall$ $X,Y,Z$ $\in$ $ \mathfrak{X}(M) $.
In a local coordinate system the latter reads
\beq
\partial_{i}g_{jk}=\Gamma_{kji}+\Gamma^{*}_{jki}\label{dual}
\eeq
This is the defining equation for the coefficients $\Gamma^{*}_{jki}$ of the dual connection $\nabla^{*}$. The two connections are then said to be dually coupled.
 \end{mydef}
 \textbf{Remark.} Note the in the definition above it is implicitly assumed that $\Gamma_{jki}=\Gamma_{(jk)i}$ (and same for the dual) since expression (\ref{dual}) is symmetric in $j,k$ and as such it only makes sense  if both connections are symmetric in their first pair of indices. If the connections did not have this symmetry, equation (\ref{dual}) would fail to completely define the full dual connection (i.e. it would only define the symmetric part, in the first two indices, of the dual connection coefficients and leave the antisymmetric completely unspecified). If, in addition, the two connections are torsion-free we have the following result which can be found for instance in \cite{amari1985differential,amari1987differential,lauritzen1987statistical}.
 \begin{myprop}
     If, in addition, $\nabla$ and $\nabla^{\star}$ are torsion-free then there exists a totally symmetric tensor $C_{ijk}=C_{(ijk)}$ called the cubic tensor \cite{amari1985differential,amari1987differential} such that 
       \beq
\nabla_{i}g_{jk}=C_{ijk} \;\; , \;\; \nabla_{i}^{*}g_{jk}=-C_{ijk} \;\;, \;\; with \; \;C_{ijk}=C_{(ijk)}
         \eeq
 \end{myprop}
     \begin{proof}
         Even though the proof is quite standard and can be found in any work dealing with the subject, we will  prove this fact here with the use of post-Riemannian expansion, which as we have already mentioned, is something that is not frequently used in the mathematics literature but is quite standard in the modified gravity community.  Using the aforementioned post-Riemannian expansions both for the connection and its dual, we have
         \beq
\Gamma_{kji}=\Gamma_{kji}^{(0)}+N_{kji}
         \eeq
         \beq
\Gamma_{jki}^{*}=\Gamma_{jki}^{(0)}+N_{jki}^{*}
         \eeq
         Recall that $\Gamma_{jki}^{(0)}$ is the Levi-Civita connection satisfying $\partial_{i}g_{jk}=\Gamma_{kji}^{(0)}+\Gamma^{(0)}_{jki}$. Using this fact and substituting both of the above into the definition (\ref{dual}) it follows that
         \beq
0=N_{kji}+N_{jki}^{*}
         \eeq
meaning that the two distortions annihilate each other and, in addition, only the symmetric part in the first two indices contributes (see discussion in previous section), so we may set $C_{kji}=C_{(kj)i}:=-2N_{kji}$. Then, from torsion freeness  $S_{ijk}=N_{k[ij]}=0$, which, along with the symmetry of $C_{kij}$ in the first two indices implies that the latter is totally symmetric, viz. $C_{kji}=C_{(kji)}$. Therefore, given also the fact that $Q_{ijk}=-2N_{(jk)i}=-2N_{jki}= C_{jki}$ we find
 \beq
\Gamma_{kji}=\Gamma_{kji}^{(0)}-\frac{1}{2}Q_{kji}=\Gamma_{jki}^{(0)}-\frac{1}{2}C_{kji} \label{T1}
         \eeq
          \beq
\Gamma_{kji}^{*}=\Gamma_{kji}^{(0)}+\frac{1}{2}Q_{kji}=\Gamma_{jki}^{(0)}+\frac{1}{2}C_{kji} \label{T2}
         \eeq
         where $C_{ijk}=Q_{ijk}=-2 N_{ijk}$ is a rank-3 tensor that is totally symmetric in its indices. This tensor, has been given various names. It was dubbed the "skewness" tensor by Lauritzen \cite{lauritzen1987statistical} and the "cubic" tensor by Amari \cite{amari1985differential,amari1987differential}. here we see that this is nothing else than (a totally symmetric) non-metricity. It is rather surprising that this fact is seldom mentioned in any of the relevant works in the mathematics literature.  It also immediately follows that
         \beq
\nabla_{i}g_{jk}=C_{ijk} \;\; , \;\; \nabla_{i}^{*}g_{jk}=-C_{ijk} \;\;, \;\; with \;\;C_{ijk}=C_{(ijk)}
         \eeq
         and that the mean connection is just the Levi-Civita one:
\beq
\frac{1}{2}(\nabla+\nabla^{*})\equiv \nabla^{(0)}
\eeq  
     \end{proof}
With this last result we see that the concept of the dual connection becomes more transparent, since given one connection $\nabla$ we can always find another one $\nabla^{*}$ and by forming an appropriate combination of them produce a net connection that is metric compatible. Closely related to the dual connections concept is that of Statistical Manifolds \cite{amari1987differential} which we briefly review below.

\begin{mydef}
Statistical Manifolds. Let $g$ be a Riemannian metric and $\nabla$ a torsion free connection over the differentiable manifold $\mathcal{M}$. Then, if there exists a totally symmetric tensor such that $\nabla g=C$, the triple (g,$\nabla$,$C$,$\mathcal{M}$) is called a statistical manifold \cite{amari1985differential}. In addition, if $\nabla^{*}$ is the dual connection given by equation (\ref{metdual}), then the collection(g,$\nabla^{*}$,$C$,$\mathcal{M}$) is referred to as dual statistical manifold.

\end{mydef}

Note: Let us mention that the term "Statistical" does not mean that these notions are tied only to statistics. These are purely geometrical constructions which however were first applied in statistics \cite{amari1997information}  .  The following  typical example illustrates this connection.

\subsection{Example.}

    Consider a family of probability distributions $p=p(x,\xi)\geq0$, $\int p(x,\xi)dx=1$. Then given that $p(x,\xi)$ varies smoothly on the parameter space $\xi=(\xi_{1},\xi_{2},...,,\xi_{m})$, these probability distributions define differential manifolds.
    In particular, we have an induced metric, the so-called Fisher metric with components defined by\footnote{This metric measures 'distances' between distributions and was considered by Rao \cite{rao1945information}.}
    \beq
g_{ij}(\xi):=E_{\xi}[\partial_{i}l,\partial_{i}l ]=\int p(x,\xi)\partial_{i}l(x,\xi),\partial_{i}l(x,\xi) dx 
    \eeq
where $E_{\xi}[f ]=\int f(x,\xi)p(x,\xi)dx$ denotes the expectation value of a function $f$ and $l(x,\xi):=\ln{p(x,\xi)}$ is the so-called log-likelihood function. We have also abbreviated $\partial_{i}=\frac{\partial}{\partial \xi^{i}}$.

    Furthermore a completely symmetric non-metricity (cubic tensor) is also induced as
    \beq
C_{ijk}(\xi):= E_{\xi}[\partial_{i}l,\partial_{j}l,\partial_{k}l ]=\int p(x,\xi)\partial_{i}l(x,\xi),\partial_{i}l(x,\xi) \partial_{k}l(x,\xi)dx 
    \eeq

    As a classical example we consider the Gaussian distribution 
    \beq
    p(x,\mu,\sigma)=\frac{1}{\sqrt{2 \pi}\sigma}e^{-\frac{(x-\mu)^{2}}{2 \sigma^{2}}}
    \eeq
    where now $(\xi_{1},\xi_{2})=(\mu,\sigma) \in \mathbb{R}x\mathbb{R}^{+}$ namely a $2$-dimensional manifold is defined with coordinates $\mu$ (the mean) and $\sigma$ (the variant). For this case we easily compute $g_{11}=\frac{1}{\sigma^{2}}$, \; $g_{22}=\frac{1}{2 \sigma^{2}}$ and $g_{12}=g_{21}=0$ and therefore the Fisher metric \cite{fisher1922mathematical} reads
 \beq
    g_{ij}=\frac{1}{\sigma^{2}}diag(1,2)
    \eeq

    Furthermore, the components of the cubic tensor are in this case
    \beq
C_{111}=C_{122}=C_{212}=C_{221}=0 \nonumber
    \eeq
     \beq
C_{112}=C_{121}=C_{211}=\frac{2}{\sigma^{3}} \;,\;\; C_{222}=\frac{8}{\sigma^{3}}\nonumber
    \eeq
We see therefore how such a dualistic structure arises from statistics.

\subsection{The $\alpha$-Connections}

Let us now review the notion of the so-called $\alpha$-connections \cite{amari1987differential}. The starting point is to observe that since $C_{ijk}$ are the components of a totally symmetric tensor, the quantities $\alpha C_{ijk}$, $\alpha \in \mathbb{R}$ also define the components of a totally symmetric tensor. As a result, for any pair of conjugate connections ($\nabla, \nabla^{\star}$), the one-parameter family of connections $ \{ \nabla^{(\alpha)}\}_{\alpha \in \mathbb{R}}$ defines a pair $(\nabla^{(\alpha)},\nabla^{(-\alpha)})$ which is also dually coupled to the metric. In components we have the following monoparametric  family of connections
\beq
\nabla^{(\alpha)}:=\nabla^{(0)}-\frac{\alpha}{2}C
\eeq
with coefficients\footnote{Note that $\Gamma_{kij}=g_{kl}\Gamma^{l}_{\;\; ij}$}
\beq
\Gamma_{kij}^{(\alpha)}=\Gamma_{kij}^{(0)}-\frac{\alpha}{2}C_{kij}
\eeq
\beq
\Gamma_{kij}^{(-\alpha)}=\Gamma_{kij}^{(0)}+\frac{\alpha}{2}C_{kij}
\eeq
Note also that $\nabla^{(0)}$ is the Levi-Civita connection and that $\Gamma_{kij}^{(\alpha)}$ can be computed directly by $\Gamma_{kij}$ and $\Gamma_{kij}^{\star}$ from  the formula
\beq
\Gamma_{kij}^{(\alpha)}=\frac{(1+\alpha)}{2}\Gamma_{kij}+\frac{(1-\alpha)}{2}\Gamma_{kij}^{\star}
\eeq
or in a coordinate free manner
\beq
\nabla^{(\alpha)}=\frac{(1+\alpha)}{2}\nabla+\frac{(1-\alpha)}{2}\nabla^{\star}
\eeq
which is trivially proved by using the above expressions and the relation (\ref{T1}) in order to eliminate the Levi-Civita connection and the cubic tensor in favor of the full connection and its dual.. Furthermore, it obviously holds that
\beq
\frac{1}{2}(\nabla^{(\alpha)}+\nabla^{(-\alpha)})\equiv \nabla^{(0)} \label{alpha}
\eeq
confirming the fact that these connections are dually coupled to the metric as stated earlier.

\subsection{On Possible Generalizations and Equivalent Interpretations}
As we already mentioned the dual connection definition (\ref{dual}) can be generalized in many ways. One is to allow for torsion. Another one is to release the assumption that the connection coefficients  are symmetric in respect to their first two indices. As we shall show bellow another generalization consists in extending the dual connection to a one-parameter space of dual connections. Our starting point is the fact that the convex linear combination
\beq
t\nabla +(1-t)\nabla^{*} \;\; t \in \mathbb{R}
\eeq
is a connection. We can therefore generalize  the dual connection definition to
\beq
\partial_{i}g_{jk}=2 t \Gamma_{(kj)i}+2(1-t)\Gamma_{(kj)i}^{*}
\eeq
For the special case $t=1/2$ and $\Gamma_{(ij)k}=\Gamma_{ijk}$ this coincides with (\ref{metdual}). It is then trivial to see  that 
\beq
\nabla^{(0)}=t\nabla +(1-t)\nabla^{*} \;\; t \in \mathbb{R}
\eeq
that is the convex linear combination gives the Levi-Civita connection. However, with this generalization the basic property $(\nabla^{*})^{*}=\nabla$ is lost.  In fact it is not difficult to show that the condition $(\nabla^{*})^{*}=\nabla$ fixes $t=1/2$ . As a result, from the above parametric family only the member with $t=1/2$ has the property that the double dual connection gives back $\nabla$ (i.e. duality is involutive). Consequently, we see  that the prototype definition (\ref{dual}) of the dual is indeed unique being the only one with this properly. As a final remark on the concept of dual (non-metric) connection as given by eq. (\ref{dual}) let us also present an another point of view that yields to the same result. To this end, let us consider two affine connections $\nabla^{(1)}$ and $\nabla^{(2)}$ on $\mathcal{M}$.  For all $X,Y,Z$ $\in$ $\mathfrak{X}(M)$ we define the $(0,3)$ tensor field
\beq
W(Z,X,Y):=\frac{1}{2}(\nabla^{(1)}_{Z}g+\nabla^{(2)}_{Z}g)(X,Y)= \frac{1}{2}(\nabla_{Z}^{(1)}+\nabla_{Z}^{(2)})g(X,Y) \label{mutnon}
\eeq
which we may call the $mutual$ $no-metricity$ tensor. In local coordinates it reads
\beq
W_{ijk}=\partial_{i}g_{jk}-\Gamma^{(1)}_{(jk)i}-\Gamma^{(2)}_{(jk)i}
\eeq
Then, it is obvious that given $\nabla^{1}=\nabla$ the demand of a  vanishing mutual non-metricity (i.e. $W_{ijk}=0$) implies (setting also $\nabla^{(2)}=\nabla^{\star}$)
\beq
\partial_{i}g_{jk}= \Gamma_{(kj)i}+\Gamma_{(kj)i}^{*}
\eeq
which is exactly the defining relation of the dual connection. As a result one can interpret the dual (non-metric)  connection $\nabla^{\star}$, as the one for which the mutual non-metricity vanishes. So, condition (\ref{dual}) is equivalent to the statement
\beq
W(Z,X,Y)=\frac{1}{2}(\nabla_{Z}g+\nabla^{\star}_{Z}g)(X,Y)\equiv0 \label{mutnon2}
\eeq
which can then be seen as the definition of the dual connection. 

Let us also mention that one can also relax the torsionlessness assumption. Such dual connections admitting torsion are also very interesting since they are encountered in quantum statistical models \cite{amari2000methods}\footnote{In particular, the non-commutativity of quantum operators generates a non-vanishing torsion.}. However our intention here is not to go over this direction but rather develop a dual connection  analogue for connection that are metric but admit torsion. This can be seen as the complementary subspace of the space of connections that are dual in both senses. We make our idea more precise in the following section.

\section{The dual analogue of two metric but torsionful connections}\label{sec}

 The dual connection concept we discussed earlier has the remarkable property to keep the inner product intact under parallel transport regardless of the fact that both connections have a non-vanishing non-metricity. To quote Amari's words "the two connections cooperate to keep the inner product unchanged under parallel transport". 

 In this sense it would be interesting to have a similar concept of a metric but torsionful connection and its corresponding dual in such a way so as to preserve vectors from twisting under parallel transport. To make our idea more clear let us recall the effect of torsion with a simple geometrical illustration.

	Recall that in the presence of torsion we cannot form small parallelograms by parallel transportation of one vector to the direction of the other and vice versa. The end result is a pentagon.  To see this  we borrow the following illustration from \cite{iosifidis2019metric}.   On a manifold endowed with an affine connection $\nabla$ consider two curves with parametrizations $\mathcal{C}:x^{i}=x^{i}(\lambda)$ and $\mathcal{\tilde{C}}:\tilde{x}^{i}=\tilde{x}^{i}(\lambda)$ and associated tangent vectors
	\begin{equation}
	u^{i}=\frac{dx^{i}}{d\lambda}
\;\; and \;\; 
	\tilde{u}^{i}=\frac{d\tilde{x}^{i}}{d\lambda}
	\end{equation}
	to each respectively. Now, let us $\delta \lambda$-displace $u^{i}$ along $\mathcal{\tilde{C}}$ to obtain $u^{'i}$ which to first order in $\delta \lambda$ is given by
	\begin{equation}
	u^{'i}=u^{i}+(\partial_{j}u^{i})\frac{d\tilde{x}^{j}}{d \lambda}\delta \lambda +\mathcal{O}(\delta \lambda^{2} ) \label{toru}
	\end{equation}
	but since $u^{i}$ is parallely transported along $\mathcal{\tilde{C}}$, it holds that
	\begin{equation}
	\frac{d\tilde{x}^{j}}{d\lambda}\nabla_{j}u^{i}=0= \frac{d\tilde{x}^{j}}{d\lambda}\partial_{j}u^{i}+\Gamma^{i}_{\;\;\;jk}\frac{d\tilde{x}^{k}}{d\lambda}u^{j} 
	\end{equation}
	which when combined with $(\ref{toru})$ results in
	\beq
	u^{'i}=u^{i}-\Gamma^{i}_{\;\;\;j k}u^{j}\tilde{u}^{k}\delta\lambda
	\eeq
	Doing the same job but now  displacing   $\tilde{u}^{i}$ along $\mathcal{C}$, we get
	\beq
	\tilde{u}^{'i}=\tilde{u}^{i}-\Gamma^{i}_{\;\;\;jk}\tilde{u}^{j}u^{k}\delta\lambda +\mathcal{O}(\delta \lambda^{2} ) = \tilde{u}^{i}-\Gamma^{i}_{\;\;\;jk}\tilde{u}^{j}u^{k}\delta\lambda +\mathcal{O}(\delta \lambda^{2} ) 
	\eeq
	Subtracting the latter two, it follows that
	\beq
	(\tilde{u}^{i}+u^{'i})-(u^{i}+\tilde{u}^{'i})= T^{i}_{\;\; kj}\tilde{u}^{j}u^{k}\delta\lambda
	\eeq
	Notice now that for the infinitesimal parallelogram to exist, the vectors $(\tilde{u}^{i}+u^{'i})$ and $(u^{i}+\tilde{u}^{'i})$ should be identical and as it is clear from the above, this is not true in the presence of torsion. Defining the vector that shows this deviation as $V^{i}\delta \lambda=(\tilde{u}^{i}+u^{'i})-(u^{i}+\tilde{u}^{'i})$ the latter can also be written as\footnote{This only holds true for small displacements in the directions of $\tilde{u}^{i}$ and $u^{i}$ which themselves are computed at the starting point of the path.}
	\beq
	V^{i}=T^{i}_{\;\; kj}\tilde{u}^{j}u^{k}
	\eeq
	which is the vector that shows how much the parallelogram has been deformed. 

 Now consider the same setting but apart from $\nabla$ we introduce another connection $\nabla^{\star}$. We parallel transport $u^{i}$ infinitesimally  along $\mathcal{\tilde{C}}$ using the connection $\nabla$ to obtain $u^{'i}$. Then we parallel transport $\tilde{u}^{i}$ infinitesimally  along $\mathcal{C}$  but now using the connection $\nabla^{\star}$ instead, to end up with $\tilde{u}^{'i}$. Similarly to the above computation, finally we find for the deviation vector:
 \beq
(\tilde{u}^{i}+u^{'i})-(u^{i}+\tilde{u}^{'i})=\Big( \Gamma^{i}_{\;\; jk}-\Gamma^{i \; \star}_{\;\; kj}\Big) \tilde{u}^{j}u^{k}
 \eeq
	Observe now that if the connection coefficients of the star connection satisfy
 \beq
\Gamma^{i \; \star}_{\;\; kj}=\Gamma^{i}_{\;\; jk} \label{dutor0}
 \eeq
they preserve the breaking of the parallelogram! Namely, the two connections, cooperate in order to preserve the infinitesimal parallelograms from breaking to pentagons, even though the space has torsion! In analogy with the introduction of the dual connection which helps in preserving inner products, we have now come to the notion of a \underline{ Torsion Dual Connection} which may be used in order to preserve the formation of infinitesimal parallelograms. Then, equation (\ref{dutor0}) is the defining relation of this dual torsion connection. Collecting everything we may state the following.
\begin{cor}
Consider a (torsionful) connection $\nabla$ on $\mathcal{M}$. Introduce now another (torsionful) connection $\nabla^{\star}$ on $M$  with coefficients given by (\ref{dutor0}). Then given two vectors, the parallel transport of one in the direction of the other using $\nabla$ and the same for the other but now using $\nabla^{\star}$ preserves the vectors from twisting. 
\end{cor}

\textbf{Remark.} A quite remarkable consequence of the inclusion of a second connection is the following. Consider the collection ($g,\nabla^{(1)},\nabla^{(2)}, \mathcal{M}$) where the connections $\nabla^{(1)}$ and  $\nabla^{(2)}$ are  both torsion-free but totally independent from each other. Furthermore, consider the above situation with the two curves $\mathcal{C}$ and $\mathcal{\tilde{C}}$ and the associated tangent vectors. We now transport $u^{i}$ using $\nabla^{(1)}$ and $\tilde{u}^{i}$ using $\nabla^{(2)}$. Then, we find the non-vanishing deviation
\beq
(\tilde{u}^{i}+u^{'i})-(u^{i}+\tilde{u}^{'i})=\Big( \Gamma^{i\; (1)}_{\;\; jk}-\Gamma^{i \;\; (2)}_{\;\; kj}\Big) \tilde{u}^{j}u^{k} \neq 0
\eeq
 meaning that infinitesimal parallelograms break to pentagons even though the space if free of torsion! We see therefore that it is possible to have vectors twisted under parallel transport even in the absence of torsion, given that we introduce a second connection on the manifold. From this observation it stands to reason to call the tensor\footnote{This tensor also naturally appears by acting the operator $\nabla^{(1)}_{i}\nabla_{j}^{(2)}-\nabla^{(2)}_{j}\nabla_{i}^{(1)}$ on a scalar function. In particular, it holds that $\Big(\nabla^{(1)}_{i}\nabla_{j}^{(2)}-\nabla^{(2)}_{j}\nabla_{i}^{(1)}\Big)f=-M^{l}_{\;\;ji}\partial_{l}f$. } 
 \beq
 M^{i}_{\;\; jk}:=\Gamma^{i\; (1)}_{\;\; jk}-\Gamma^{i \;\; (2)}_{\;\; kj}
 \eeq
 the mutual  torsion tensor (see (\cite{puechmorel2020lifting})), which we may also express in a coordinate-free fashion as
 \beq
 M(Y,X)=\nabla^{(1)}_{X}Y-\nabla^{(2)}_{Y}X-[X,Y] \label{muttor}
 \eeq
 This tensor is of course defined for all connections (including torsion). Note that for the special case of torsion-free connections, this tensor coincides with the difference tensor given by eq. (\ref{diften})\footnote{Quick Proof: Let $T^{(1)}(X,Y)$ and $T^{(2)}(X,Y)$ be the torsion tensors of the two connections. Adding a zero to the definition of the mutual torsion and regrouping the remaining terms we readily get $M(Y,X)=\nabla_{X}^{(1)}Y-\nabla_{X}^{(2)}Y+\nabla_{X}^{(2)}Y-\nabla^{(2)}_{Y}X-[X,Y]=K(Y,X)+T^{(2)}(X,Y)=\nabla^{(1)}_{X}Y-\nabla_{Y}^{(1)}X+\nabla_{Y}^{(1)}X-\nabla^{(2)}_{Y}X-[X,Y]=T^{(1)}(X,Y)+K(X,Y)$ from which it follows immediately that M coincides with K for vanishing torsions and, furthermore, they are also symmetric in this case $M(X,Y)=M(Y,X)=K(X,Y)$.}. Therefore, even for torsion-free connections, their mutual torsion may not be zero!
 
 After this side-note let us go back to the concept of dual torsion connection we introduced earlier, formally define it and gather some basic properties of the latter.
 
 \begin{mydef}
 Let $g$ be a Riemannian metric and $\nabla$ a torsionful, and in general non-metric, affine connection over the differential manifold $\mathcal{M}$. We define a second connection on $\mathcal{M},$ the  \underline{torsion-dual connection} of $\nabla$ (and denoted with $\nabla^{\star}$) as
 \beq
 \nabla_{X}Y-\nabla_{Y}^{\star}X-[X,Y]=0 \label{tordu}
 \eeq
 for all $X,Y$ in $ \mathfrak{X}(M)$. Essentially the torsion dual connection is specified by the demand of vanishing mutual torsion.
 In local coordinates  it takes the form
  \beq
\Gamma^{i \; \star}_{\;\; kj}=\Gamma^{i}_{\;\; jk} \label{dutor}
 \eeq
 
 Then the geometric object formed by parallel transporting two tangent vectors one in the direction of the other, is always a parallelogram given that one of the vectors is transported with respect to $\nabla$ and the other with respect to $\nabla^{\star}$. Note that unlike (\ref{dual}) the dual connection here is defined without the need of any metric and furthermore the above equation specifies all the components of the dual connection without having to impose more assumptions.
 \end{mydef}
 Now, there are two immediate consequences coming with the above inclusion:
 \begin{enumerate}[label=(\roman*)]
 \item If the connections $\nabla$ and $\nabla^{\star}$ are metric then the mean connection
 \beq
 \nabla^{(0)}=\frac{1}{2}(\nabla+\nabla^{\star})
 \eeq
 is the Levi-Civita connection.
     \item $(\nabla^{\star})^{\star}=\nabla$
     
 \end{enumerate}
 \begin{proof}
     Even though the proofs of the these two statements are self-evident we shall include them here for completeness. We chose a local coordinate basis and work with the connection coefficients. Regarding the first statement, we have
     \beq
    \frac{1}{2}( \Gamma^{k}_{\;\; ij}+\Gamma^{k}_{\;\; ij}*)= \frac{1}{2}( \Gamma^{k}_{\;\; ij}+\Gamma^{k}_{\;\; ji})=\Gamma^{k}_{\;\; (ij)}
     \eeq
     where we have used the defining relation (\ref{tordu}) for the torsion dual connection. The above equation implies that the torsion of the mean connection is zero. Then if the two connections are also metric-compatible, the mean connection will also be compatible with the metric and  given the uniqueness of the torsion free and metric Levi-Civita connection we conclude that the former must coincide with the latter. As a result
     \beq
\frac{1}{2}(\nabla+\nabla^{\star})\equiv  \nabla^{(0)}
 \eeq
 \end{proof}
 
 \begin{proof}
     The second statement is also self-evident since, if we see the star as an operation that maps $*:$ $\Gamma^{k}_{\;\;ij}\mapsto \Gamma^{k}_{\;\;ji}$, we easily compute
     \beq
    ( \Gamma^{k \; \star}_{\;\; ij})^{\star}=\Gamma^{k \; \star}_{\;\; ji}=\Gamma^{k}_{\;\; ij}
     \eeq
     that is
     \beq
     (\nabla^{\star})^{\star}=\nabla
     \eeq
     
 \end{proof}
 
 Continuing, we also have the following.
 
 \begin{myprop}
     Let $T(X,Y)$ and $T^{*}(X,Y)$ be the torsion corresponding to the affine connection $\nabla$ and its torsion dual $\nabla^{*}$ respectively. Then,it holds that
     \beq
     T(X,Y)+T^{*}(X,Y)=0
     \eeq
 \end{myprop}
 \begin{proof}
 This is shown  trivially by recalling the definition of the torsion tensor and the torsion dual connection. By direct calculation we have
 \begin{gather}
     T(X,Y)+T^{*}(X,Y)=\nabla_{X}Y-\nabla_{Y}X-[X,Y]+\nabla_{X}^{*}Y-\nabla_{Y}^{*}X-[X,Y]=\nonumber \\
     =\Big( \nabla_{X}Y-\nabla_{Y}^{*}X-[X,Y]\Big) -\Big( \nabla_{Y}X-\nabla_{X}^{*}Y-[Y,X]\Big)=0
 \end{gather}
 where we have used the fact that the two parentheses in the second line both vanish by virtue of (\ref{tordu}).
 \end{proof}
 Note that in local coordinates the last equality takes the form
 \beq
 S_{ijk}+S^{*}_{jik}=0 \label{SS}
 \eeq
 We can double-check the validity of the lat one by performing a post-Riemannian expansion of both connections in the defining equation (\ref{tordu}). Indeed, employing such expansion, the Levi-Civita parts cancel out and one is left with
 \beq
 N_{ijk}=N^{*}_{ikj}
 \eeq
 Then, taking the antisymmetric part in $jk$ and using the fact that $S_{ijk}=N_{k[ij]}$ (and the corresponding one for the torsion dual connection) we arrive at (\ref{SS}).

  In analogy with expressions (\ref{T1}) and (\ref{T2}) here we also have similar relations given some specific conditions. To be more precise we have the following.

  \subsection{The Theorems}
  
  \begin{theorem}
   Consider the quadruple  ($g,\nabla,\nabla^{\star},\mathcal{M}$) where $\nabla$ is an affine connection and $\nabla^{\star}$ its torsion dual as given by the defining relation (\ref{tordu}). If $\nabla$ is metric and $N_{ijk}$ (the distortion tensor of $\nabla$) is antisymmetric in the last pair of indices then there exists a 3-form (i.e. a totally antisymmetric tensor of rank-3), call it A, such that
   \beq
\Gamma_{kij}=\Gamma_{kij}^{(0)}+A_{kij} \label{A1}
         \eeq
          \beq
\Gamma_{kij}^{*}=\Gamma_{kij}^{(0)}-A_{kij} \label{A2}
         \eeq
   where $\Gamma_{jki}^{(0)}$ is the Levi-Civita connection and $A_{ijk}=A_{[ijk]}$.
  \end{theorem}
 \begin{proof}
Since by hypothesis $\nabla$ is metric we have that $\nabla_{i}g_{jk}=Q_{ijk}=-2 N_{(jk)i}=0$ implying that $N_{jki}=-N_{kji}$, namely the distortion of $\nabla$ is antisymmetric in the first pair of indices. Now given also that  $N_{ijk}$ is antisymmetric in its last pair of indices, that is  $N_{ijk}=-N_{ikj}$ it immediately follows that this is also antisymmetric in the exchange of first and third indices. Indeed, employing the above symmetry properties of $N$ we have 
\beq
N_{ijk}=-N_{jik}=N_{jki}=-N_{kji}
\eeq
which proves that $N$ is also anti-symmetric under the exchange of the first with the third index. As a result, $N$ is antisymmetric in any exchange of its indices and is therefore totally antisymmetric. Consequently,    there exists a 3-form $A=A_{ijk}dx^{i}\wedge dx^{j} \wedge dx^{k}$ such that
\beq
N_{ijk}\equiv A_{ijk}=A_{[ijk]}
\eeq
Substituting this into the decomposition (\ref{expans}) gives the stated result
  \beq
\Gamma_{kij}=\Gamma_{kij}^{(0)}+A_{kij} 
         \eeq
         Furthermore, employing the post-Riemannian expansions  of both $\nabla$ and $\nabla^{\star}$ in (\ref{dutor}) we get
         \beq
N_{ikj}^{\star}=N_{ijk}
         \eeq
         which upon using the above result also implies that $N_{ikj}^{\star}=A_{ijk}=-A_{ikj}$, and  substituting this back to the post-Riemannian expansion of $\nabla$ we complete the proof
          \beq
\Gamma_{kij}^{*}=\Gamma_{kij}^{(0)}-A_{kij} 
         \eeq
 \end{proof}
 \textbf{Remark 1.} It is worth mentioning that this result is the "dual analogue" of the totally symmetric tensor cubic tensor $C_{ijk}$. In a sense, the torsionful (but metric) case we have developed here is complementary to the non-metric (but torsioless) case that appears in information geometry \cite{amari1997information}. The totally symmetric tensor $C$ has now given its place to the totally antisymmetric tensor $A$ and their symmetries reflect the symmetry properties of torsion and non-metricity respectively.

 \textbf{Remark 2.} Note that the dimension of the manifold must be $n\geq 3$\footnote{Of course this is so because a totally antisymmetric tensor of rank $3$ identically vanishes in two dimensions.}. In particular, for $n=3$ we have that $A_{ijk}\propto \epsilon_{ijk}$ with $\epsilon_{ijk}$ being the totally antisymmetric Levi-Civita tensor. That is, in this case there exists a  function $f$ such that
\beq
A_{ijk}=f\epsilon_{ijk}
\eeq
 and so
   \beq
\Gamma_{kij}=\Gamma_{kij}^{(0)}+f \epsilon_{kij}\;\;,\;\;\;\Gamma_{kij}^{*}=\Gamma_{kij}^{(0)}-f \epsilon_{kij} 
         \eeq
 and consequently in this case the space of connections (and their duals) is parameterized by a single function\footnote{Note that if $f$ is allowed to be complex, then if, in particular, the latter is purely imaginary then dualization simply amounts to taking the complex conjugation of the connection.}.

\begin{mydef}
    Torsion Dual Manifold. Let $\mathcal{M}$ be a differentiable manifold of dimension $n\geq 3$. Introduce an affine connection $\nabla$ and its torsion dual $\nabla^{\star}$ by equation (\ref{tordu}). The collection $(\nabla, \nabla^{\star},\mathcal{M})$ will be called a \underline{torsional statistical manifold}.
\end{mydef}
 
Let us recall now that on a  usual statistical manifold if the two connections (see eq. (\ref{dual})) $\nabla$ and $\nabla^{*}$ are torsionfree, then if $\nabla$ is flat so is $\nabla^{*}$ and vice versa, namely
\beq
R(X,Y)Z=0 \leftrightarrow R^{*}(X,Y)Z=0
\eeq
It is then reasonable to ask if a similar situation arises in the case of torsion dual connections. To be more precise, given a torsion dual manifold, does flatness of $\nabla$ imply flatness of $\nabla^{*}$ and vice versa? As we show below such a statement does not hold true for dual torsion manifolds, in general, but given a simple constraint the two connections can share such a property. To prove this we shall first consider a more general situation, in order to be able to make contact with the non-metric dual connections of information geometry. We state and prove the following.

\begin{mylemma}
Let $\nabla$ and $\nabla^{\star}$ be two linear affine connections that are in general neither metric nor torsionfree. Then, if their distortion tensors differ only by a sign, that is, there exists a type (1,2) tensor N, such that
\beq
\nabla=\nabla^{(0)}+N
\eeq
and
\beq
\nabla^{\star}=\nabla^{(0)}-N
\eeq
and if the latter possesses a certain symmetry/antisymmetry in its first pair of indices, then there corresponding  curvature tensors of the two connections are in the relation
\beq
(R(X,Y)Z,W)=-(R^{\star}(X,Y)W,Z)+(1-(-1)^{p})\Big( \nabla_{X}^{(0)}g(N(Z,Y),W)-\nabla_{Y}^{(0)}g(N(Z,X),W) \Big)
\eeq
or  expressed in local coordinates
\beq
R_{ijkl}+R_{jikl}^{\star}=2\Big(1-(-1)^{p}\Big) \nabla_{[k}^{(0)}N_{|ij|l]}
\eeq
where $p=0,1$ with $p=0$ corresponding to symmetry(i.e. $N_{ijk}=N_{jik}$ ) and $p=1$ to anti-symmetry ($N_{ijk}=-N_{jik}$).

\end{mylemma}

\begin{proof}
Our starting point here is the decomposition (\ref{Riemdec}), which with all indices down reads
\beq
R_{ijkl}=R^{(0)}_{ijkl}+2 \nabla^{(0)}_{[k}N^{m}_{\;\;|ij|l]}+2 N_{in[k}N^{n}_{\;\;|j|l]} \label{e}
\eeq
This is the curvature tensor associated with the connection $\nabla$. For the torsion dual connection $\nabla^{\star}$ we simply replace $N_{ijk}=-N_{ijk}$ to get
\beq
R_{ijkl}^{\star}=R^{(0)}_{ijkl}-2 \nabla^{(0)}_{[k}N^{m}_{\;\;|ij|l]}+2 N_{in[k}N^{n}_{\;\;|j|l]}
\eeq
or exchanging $i$ and $j$,
\beq
R_{jikl}^{\star}=R^{(0)}_{jikl}-2 \nabla^{(0)}_{[k}N^{m}_{\;\;|ji|l]}+2 N_{jn[k}N^{n}_{\;\;|i|l]}
\eeq
Then adding up the latter with (\ref{e}) and also using the fact that $R^{(0)}_{ijkl}+R^{(0)}_{jikl}=0$, it follows that
\beq
R_{ijkl}+R_{jikl}^{\star}=2\Big(1-(-1)^{p}\Big) \nabla_{[k}^{(0)}N_{|ij|l]}+N_{ink}N^{n}_{\;\; jl}-N_{inl}N^{n}_{\;\; jk}+N_{jnk}N^{n}_{\;\; il}-N_{jnl}N^{n}_{\;\; ik} \label{e3}
\eeq
Now, since by hypothesis $N_{ijk}=(-1)^{p}N_{jik}$ we have that
\beq
N_{jnk}N^{n}_{\;\; il}=(-1)^{p}N_{njk}N^{n}_{\;\; il}=(-1)^{p}N^{n}_{\;\; jk}N_{nil}=(-1)^{2p}N^{n}_{\;\; jk}N_{inl}=N_{inl}N^{n}_{\;\; jk}
\eeq
as well as
\beq
N_{jnl}N^{n}_{\;\; ik}=(-1)^{p}N_{njl}N^{n}_{\;\; ik}=(-1)^{p}N_{nik}N^{n}_{\;\; jl}=(-1)^{2 p}N_{ink}N^{n}_{\;\; jl}=N_{ink}N^{n}_{\;\; jl}
\eeq
and by substituting these in (\ref{e3}) the last four terms on the right-hand side of the latter cancel out leaving us with
\beq
R_{ijkl}+R_{jikl}^{\star}=2\Big(1-(-1)^{p}\Big) \nabla_{[k}^{(0)}N_{|ij|l]} \label{Lemma}
\eeq
as stated.

\end{proof}
With this result we may now obtain our main Theorem for the Riemann tensors of a connection and its torsion dual counterpart.
\begin{theorem}
 Let $\nabla$ be a metric connection and $\nabla^{\star}$ its torsion dual as given by the defining relation (\ref{tordu}). Then, the Riemann curvature tensors of the two  connections obey the relation
 \beq
R_{ijkl}+R_{jikl}^{\star}=4 \nabla_{[k}^{(0)}N_{|ij|l]} \label{Riemrel}
\eeq
In addition, flatness of $\nabla$ does not imply flatness of $\nabla^{\star}$ but rather that the coefficients of the Riemann curvature of $\nabla^{\star}$ are given by
 \beq
R_{jikl}^{\star}=4 \nabla_{[k}^{(0)}N_{|ij|l]}
\eeq
 Equivalently, we have the association
 \beq
 R_{ijkl}^{\star}=0 \Leftrightarrow R_{ijkl}=4 \nabla_{[k}^{(0)}N_{|ij|l]}
 \eeq
 Consequently, flatness of one of  the two connection imposes flatness on the other as well iff
 \beq
 \nabla_{[k}^{(0)}N_{|ij|l]}=0
 \eeq
 
\end{theorem}
\begin{proof}
Of course the proof is now trivial using the result (\ref{Lemma}). Indeed, the situation here is a special case of the latter equation where $N_{ijk}=A_{ijk}=A_{[ijk]}$ and as such $N_{ijk}$ is antisymmetric in the exchange of its first pair of indices. Therefore, putting $p=1$ and $N_{ijk}=A_{ijk}$ we obtain the stated result
 \beq
R_{ijkl}-R_{jikl}^{\star}=4 \nabla_{[k}^{(0)}A_{|ij|l]}
\eeq
From this it is also clear that $R_{ijkl}=0$ implies $R_{ijkl}^{\star}=4 \nabla_{[k}^{(0)}A_{|ij|l]}$ and similarly that $R_{ijkl}^{\star}=0$ fixes $R_{ijkl}=4 \nabla_{[k}^{(0)}A_{|ij|l]}$
\end{proof}
Therefore, $\nabla$-flatness does not imply $\nabla^{\star}$-flatness nor the other way around. However, we do have a weaker condition on the Ricci scalars. In particular, we have the following.

\begin{theorem}
 Let $\nabla$ be a metric connection and $\nabla^{\star}$ its torsion dual as given by the defining relation (\ref{tordu}). Then, the Ricci scalars of the two connections are identical. Consequently if $\nabla$ is scalar flat (i.e. $Ric(\nabla)=0$) the so is $\nabla^{\star}$ and vice-versa. 
\end{theorem}
\begin{proof}
The proof is trivial using the result (\ref{Riemrel}). Indeed, contracting the latter with $g^{ik}g^{jl}$, using the definition of the two Ricci scalars, and the fact that $A_{ijk}$ is totally antisymmetric, we easily obtain
\beq
Ric(\nabla)=Ric(\nabla^{\star})
\eeq
which also immediately implies that if one of the Ricci scalars vanishes then so does its dual counterpart as well.
\end{proof}

\textbf{Remark.} Note that one of the fundamental Theorems of information geometry, namely the fact that flatness of $\nabla$ implies that $\nabla^{\star}$ is also flat and vice-versa, i.e. $R(X,Y)Z=0 \leftrightarrow R^{\star}(X,Y)Z=0$ , comes as a bonus from our generic relation (\ref{Lemma}). Indeed for the usual non-metric dual structure, $N_{ijk}$ is totally symmetric, meaning that $p=0$ which brings (\ref{Lemma}) to the form
 \beq
R_{ijkl}-R_{jikl}^{\star}=0
\eeq
implying that when one of $\nabla$ or $\nabla^{\star}$ is flat the other one must be flat as well.

\subsection{Generalized Dual Connection in both senses (Torsion and Non-metric Dual)}

Having developed the analogue of the torsion dual connection it is there natural to ask: Given an affine connection, is it possible to define a dual connection in both senses, that is to obey both eq. (\ref{metdual}) and (\ref{tordu}) at the same time? If we suppose that the dual connection is specified by both the aforementioned conditions, then in  local coordinates these read (see eqs (\ref{dual}) and (\ref{dutor})),
\beq
\partial_{i}g_{jk}=\Gamma_{(jk)i}+\Gamma_{(jk)i}^{\star}\label{gg}
\eeq
\beq
\Gamma_{kij}=\Gamma^{\star}_{kji}
\eeq
With a first glance it is obvious that if there are no coincidental overlaps of the above conditions, in general one would expect that imposing both at the same time will further constrain the connection $\nabla$ (and subsequently $\nabla^{\star}$). This is obvious since the first of the above equation gives $n^{2}(n+1)/2$ conditions while the second supplements us with another $n^{3}$. In other words, the torsion dual connection constraint fully specifies $\nabla^{\star}$ which will then have to be compatible with (\ref{gg}) as well. Employing a post-Riemannian expansion for both connections in the above equations we readily find the conditions
\beq
N_{(jk)i}+N_{(jk)i}^{\star}=0
\eeq
and
\beq
N_{kij}+N^{\star}_{kji}=0
\eeq
Using the latter in order to eliminate $N^{\star}$ from the former, we get
\beq
N_{i(jk)}+N_{k(ij)}=0
\eeq
which has 
\beq
N_{i(jk)}=0
\eeq
as a solution, which also implies $N_{i(jk)}^{\star}=0$. Incidentally, these last two requirements also ensure that the autoparallel curves of both $\nabla$ and $\nabla^{\star}$ coincide with the geodesic paths of the manifold.
Therefore we see that it is indeed possible to demand for a dual connection in both senses, however it goes beyond the scope of this work to study such a possibility further.

\section{On Mutual Curvature and "Dual Curvature Connection"}

In the same way we defined the mutual non-metricity and torsion of two connections $\nabla^{(1)}$, $\nabla^{(2)}$ as given by equations (\ref{mutnon}) and (\ref{muttor}) respectively, it makes sense to have an analogous definition for the curvature as well, i.e. the mutual curvature of the two connections. Indeed such definitions have appeared in the literature. In particular, the authors in \cite{puechmorel2020lifting} and \cite{calin2014geometric} have defined two  different relations for the mutual (or relative) curvature of two connections. Unfortunately, both  of these definitions are rather unsatisfactory. Let us elaborate. In \cite{puechmorel2020lifting} we find the definition of the mutual curvature of two connections to be given by
\beq
R_{\nabla_{1}\nabla_{2}}(X,Y)Z:=\nabla_{X}^{(1)}\nabla_{Y}^{(2)}Z-\nabla_{Y}^{(1)}\nabla_{X}^{(2)}Z-\nabla_{[X,Y]}^{(1)}Z \label{xdefRR}
\eeq
From this definition one can easily check that $R_{\nabla_{1}\nabla_{2}}(X,Y)Z$ is not $\mathcal{F}$-linear in $Z$ as a simple calculation reveals
\beq
R_{\nabla_{1}\nabla_{2}}(X,Y)fZ:=f R_{\nabla_{1}\nabla_{2}}(X,Y)Z +(Yf) (\nabla_{X}^{(1)}Z-\nabla_{X}^{(2)}Z)+(Xf) (\nabla_{Y}^{(2)}Z-\nabla_{Y}^{(1)}Z)\label{xdefRR}
\eeq
since the last two terms spoil $\mathcal{F}$-linearity in $Z$ the above definition is not a tensor and therefore problematic.

Another definition for the mutual curvature appears in the book \cite{calin2014geometric}, page $242$ eq. ($8.9.30$). There the authors first introduce the notion of $\alpha$-connections in the usual sense (see eq. (\ref{alpha}))
and subsequently for two connections $\nabla^{(\alpha)}$ and $\nabla^{(\beta)}$ they define the mutual (relative) curvature of these two connections, by\footnote{We also correct a misprint appearing there in passing from the first to the second line of the same equation.} 
\beq
R^{(\alpha,\beta)}(X,Y)Z:=\nabla_{X}^{(\alpha)}\nabla_{Y}^{(\beta)}Z-\nabla_{Y}^{(\beta)}\nabla_{X}^{(\alpha)}Z-\nabla_{[X,Y]}^{(\alpha)}Z \label{x2}
\eeq
where now the two connections are parametrized by $\alpha$ and $\beta$ respectively.
 This already looks problematic since it is not antisymmetric in $X,Y$. In addition it singles out the connection $\nabla^{(\alpha)}$ (as was the case for (\ref{xdefRR}) as well). The authors then go on to show that this is $\mathcal{F}$-linear in $X$ which is indeed the case. Then, the authors state (without proof) that it is also $\mathcal{F}$-linear in $Y$ and $Z$. However, even though it is $\mathcal{F}$-linear in $Z$, it fails to be $\mathcal{F}$-linear in $Y$ since a trivial calculation reveals  that\footnote{The very fact that this definition is not antisymmetric in $X,Y$ requires to check the $\mathcal{F}$-linearity in $Y$ separately. In other words $\mathcal{F}$-linearity in $Y$ is not guaranteed by the $\mathcal{F}$-linearity in $X$.} 
\beq
R^{(\alpha,\beta)}(X,fY)Z=fR^{(\alpha,\beta)}(X,Y)Z+(Xf)(\nabla_{Y}^{(\beta)}-\nabla_{Y}^{(\alpha)})Z
\eeq
and obviously the presence of the last term in the above spoils the $\mathcal{F}$-linearity in $Y$.
Therefore using either of the definitions (\ref{xdefRR}) or (\ref{x2}) one still gets a result that is not $\mathcal{F}$-linear in all its arguments and therefore not a true tensor!
 As a result we see that both of the definitions of the mutual curvature tensor that appear in the literature fail to satisfy the tensor criterion. 

We shall now give our definition for the mutual curvature tensor which as we will show, not only  is it $\mathcal{F}$-linear in all its arguments (i.e. it is a tensor) but also features some additional especially appealing characteristics. We have the following.
\begin{mydef}
    Given two connections $\nabla^{(1)}$ and $\nabla^{(2)}$ we define their \underline{mutual curvature tensor} by
\beq
R_{\nabla^{1}\nabla^{2}}(X,Y)Z:=\frac{1}{2}\Big( \nabla_{X}^{(1)}\nabla_{Y}^{(2)}Z-\nabla_{Y}^{(1)}\nabla_{X}^{(2)}Z+\nabla_{X}^{(2)}\nabla_{Y}^{(1)}Z-\nabla_{Y}^{(2)}\nabla_{X}^{(1)}Z-\nabla_{[X,Y]}^{(1)}Z-\nabla_{[X,Y]}^{(2)}Z \Big) \label{Riemfor2}
\eeq
\end{mydef}
The latter is a true tensor since it is $\mathcal{F}$-linear in all its arguments as can be easily checked. In addition to that, the above definition has a number of appealing properties. Firstly, it is clearly antisymmetric in the exchange $X \Leftrightarrow Y$ as the usual Riemann tensor of a single connection. Secondly, it is also symmetric in the exchange of the two connections  $\nabla^{(1)}\Leftrightarrow \nabla^{(2)}$ and therefore it does not discriminate in either of them (i.e. the two connections are placed on equal footing). Lastly, if the two connections coincide, the mutual curvature tensor reduces to the usual Riemann tensor of the single connection $\nabla^{(1)}=\nabla^{(2)}=\nabla$. Note that the last property was also true for the mutual non-metricity and  torsion tensors as well (recall definitions (\ref{mutnon}) and (\ref{muttor})). As a result, from many different perspectives, and most importantly for consistency reasons our definition (\ref{Riemfor2}) seems to be the appropriate definition for the notion of the mutual curvature tensor. In local coordinates it reads
\beq
R^{m}_{\nabla^{1}\nabla^{2}\; kij} = \partial_{[i}\Gamma^{m\;(2)}_{\;\;\;|k|j]}+ \partial_{[i}\Gamma^{m\;(1)}_{\;\;\;|k|j]}+\Gamma^{l \;(2)}_{\;\; k[j}\Gamma^{m\; (1)}_{\;\; |l|i]}+\Gamma^{l \;(1)}_{\;\; k[j}\Gamma^{m\; (2)}_{\;\; |l|i]}
\eeq
By adding a zero and regrouping the various terms it can be brought to the alternative form
\beq
R_{(\nabla^{1}\nabla^{2})\; kij}^{m}=\frac{1}{2}\Big(R_{(\nabla^{1})\; kij}^{m}+R_{(\nabla^{2})\; kij}^{m}\Big)+K_{\;\; k[i}^{l}K^{m}_{\;\; |l|j]}
\eeq
which now clearly manifests its tensorial nature (being the sum of three type $(1,3)$ tensors). Note again that this is indeed symmetric on the exchange of the two connections since the difference tensor appears in a quadratic combination given in the last term of the above. Let us observe that the above expression is similar to the curvature tensor of the mean connection as given by eq. (\ref{Riem2}) for $t=1/2$, however there is a crucial sign difference in the last term that is quadratic in the difference tensor.

Now, in analogy with the notions of non-metric dual and torsion dual connections as given by eqs (\ref{dual}) and (\ref{tordu}) we may define what we shall call the \underline{"curvature dual connection"}. To arrive at this notion let us remark that both the non-metric and torsion dual connections are defined by imposing a vanishing mutual non-metricity and vanishing mutual torsion respectively (see eqns (\ref{mutnon2}) and (\ref{tordu})). In the same we we may define the notion of a curvature dual connection by demanding a vanishing mutual curvature tensor. With the above considerations we are led to the following definition.

\begin{mydef}
    Let $\nabla$ be an affine connection on $\mathcal{M}$. The \underline{"Curvature Dual Connection"}, denoted by $\nabla^{\star}$, is defined as that connection for which the mutual curvature tensor formed by  the two connections identically vanishes, i.e.
    \beq
R_{\nabla \nabla^{*}}(X,Y)Z\equiv 0
    \eeq
    \end{mydef}
From (\ref{Riemfor2}) for $\nabla^{(1)}=\nabla$ and $\nabla^{(2)}=\nabla^{*}$ the above constraint reads
\beq
\nabla_{X}\nabla_{Y}^{*}Z-\nabla_{Y}\nabla_{X}^{*}Z+\nabla_{X}^{*}\nabla_{Y}Z-\nabla_{Y}^{*}\nabla_{X}Z-\nabla_{[X,Y]}Z-\nabla_{[X,Y]}^{*}Z =0
\eeq
which is the defining relation for the curvature dual connection $\nabla^{*}$. Let us remark that unlike the cases of non-metric and torsion dual connections (eqns (\ref{dual}) and (\ref{tordu}) respectively), here the expression of $\nabla^{*}$ in terms of $\nabla$ is not a simple algebraic one but rather it is a partial differential equation! This is seen more clearly when expressed  in local coordinates, in which the latter equation assumes the form
\beq
\partial_{[i}\Gamma^{m\;\star}_{\;\;\;|k|j]}+ \partial_{[i}\Gamma^{m}_{\;\;\;|k|j]}+\Gamma^{l \;\star}_{\;\; k[j}\Gamma^{m}_{\;\; |l|i]}+\Gamma^{l }_{\;\; k[j}\Gamma^{m\; \star}_{\;\; |l|i]}=0 \label{curvdu}
\eeq
Then, given $n^{3}$ functions $\Gamma^{k}_{\;\;ij}$ the above is a set of partial differential equations for the functions $\Gamma^{k \; *}_{\;\;ij}$. It would be quit interesting to find the solution to (\ref{curvdu}) even for some simple cases. Of course in general such a task is highly non-trivial. Even of greatest interest would then be to see how such curvature dual connections appear in statistics or even in physical situations.

\textbf{Remark.} Relation (\ref{Riemfor2}) also shows that even if both connections $\nabla^{(i)}$,$i=1,2$ are flat, they can still exhibit a non-vanishing mutual curvature, namely with respect to the latter the manifold is not flat!

\section{Conclusions}
We have extended (and developed) the dual connection analogue of non-metric and symmetric connections to that of metric but torsionful connections. In particular, after briefly reviewing  the basic results for the non-metric (and torsion-free) dual connection, its geometrical importance and its applications in statistics, we extended this notion to a metric but torsionful dual connection. To be more precise, we defined the notion of a \underline{Torsion Dual Connection} as follows. Given an affine connection $\nabla$ on a manifold, its  torsion dual is defined as the one connection (denoted by $\nabla^{\star}$) which when combined with $\nabla$ preserves the formation of parallelograms when two vectors are transported one in the direction of the other. In other words, just as the dual non-metric connection helps to preserve inner products, here the torsion dual connection "cooperates"with $\nabla$ in order to preserve the vectors from twisting when transported over the manifold. The defining relation for this torsion dual connection is eq. (\ref{tordu}) and can be easily remembered as the one (and only one) connection we should pick in order for the mutual torsion tensor of $\nabla$ and $\nabla^{\star}$ to be vanishing!  

In addition, we have  proved two basic Theorems associated with the dualistic space that consist of a connection and its torsion dual. To be more precise, we proved that if $\nabla$ and its torsion dual $\nabla^{\star}$ are metric then there exists a $3-form$ $A$ such that the connection coefficients and its dual counterparts are given by equations (\ref{A1}) and (\ref{A2}) respectively. For this geometric structure we also proved two basic properties: i) that the mean connection $1/2(\nabla+\nabla^{\star})$ is the Levi-Civita connection and that ii) Conjugation is involutive in the sense that $(\nabla^{\star})^{\star}$ holds true for this toirsion dual connection. Furthermore, we have obtained a result analogous to the fundamental Theorem of information geometry, but now for dual torsion manifolds. Specifically we showed that $\nabla$-flatness does not imply that the torsion dual connection $\nabla^{\star}$ is also flat but rather that its curvature takes the form $R_{jikl}^{\star}=4 \nabla_{[k}^{(0)}A_{|ij|l]}$. On the other hand, we have obtained a weaker condition on the geometries and in particular we have found that if $\nabla$ is scalar-flat then so is $\nabla^{\star}$ and vice-versa.

Finally, we proposed a novel  and self-consistent definition for the notion of mutual curvature of two connections $\nabla^{(1)}$ and $\nabla^{(2)}$. Contrary to other definitions found in the literature our definition (\ref{Riemfor2}) yields a results that is $\mathcal{F}$-linear in all its arguments, namely a true tensor. Furthermore, this definition comes with two additional bonus features; it does not discriminate among the two connections and is also antisymmetric in its inputs as the usual curvature of a single connection. With the definition of the mutual curvature tensor at hand, we also defined what we call a \underline{Curvature Dual Connection} which is obtained by demanding a  vanishing mutual curvature among the two connections.  The components of the curvature of such dual connection $\nabla^{\star}$ are given by eq. (\ref{curvdu}) which is analogous to the classical dual (non-metric) connection condition and torsion dual connection (\ref{tordu}) but with the key difference that  the connection components of the curvature dual connection are now given not by an algebraic but rather by a differential equation. It remains now to see how these new notions of torsion and curvature dual connections come into play in information geometry or some other physical situations.

\section{Acknowledgements}

 This work was supported by the Estonian Research Council grant (SJD14).
 I would like to the Prof. Steve Rosenberg for some useful email discussions.

	\bibliographystyle{unsrt}
	\bibliography{ref}

\end{document}